\documentclass{amsart}
\usepackage{verbatim}
\usepackage{amsmath,amssymb,amsthm}
\usepackage[sorted]{amsrefs}
\usepackage{todonotes}
\usepackage{hyperref}

\newcommand{\BQ}{B_{\QQ}}
\newcommand{\BQCI}{B_{\QQ}^{\text{c.i.}}}

\newcommand{\kk}{\Bbbk}

\newcommand{\QQ}{\mathbb Q}
\newcommand{\Qpos}{\QQ_{\geq 0}}
\newcommand{\VV}{\mathbb V}
\newcommand{\Vpos}{\VV_{\geq 0}}
\newcommand{\ZZ}{\mathbb Z}
\newcommand{\Zpos}{\ZZ_{\geq 0}}
\newcommand{\UU}{\mathbb{U}}
\newcommand{\Upos}{\UU_{\geq 0}}

\renewcommand{\vec}[1]{\mathbf{#1}}

\DeclareMathOperator{\codim}{codim}
\DeclareMathOperator{\coker}{coker}

\DeclareMathOperator{\pdim}{pdim}
\DeclareMathOperator{\reg}{reg}

\newtheorem{theorem}{Theorem}
\newtheorem{prop}[theorem]{Proposition}
\newtheorem{lemma}[theorem]{Lemma}
\newtheorem*{question}{Question}
\newtheorem{corollary}[theorem]{Corollary}

\theoremstyle{definition}
\newtheorem{example}[theorem]{Example}
\newtheorem{remark}[theorem]{Remark}
\newtheorem{algo}[theorem]{Algorithm}

\newcommand{\defi}[1]{\textbf{\textit{#1}}}

\title{Rational combinations of Betti diagrams of complete intersections}
\author[Annunziata]{Michael T. Annunziata}
\address{Mike Annunziata, Wake Forest University, 2866 Quincy Drive, Winston-Salem, NC 27106}
\email{annumt12@wfu.edu}

\author[Gibbons]{Courtney R. Gibbons}
\address{Courtney Gibbons, Department of Mathematics, Hamilton College, 198 College Hill Road, Clinton, NY 13323}
\email{crgibbon@hamilton.edu}

\author[Hawkins]{Cole Hawkins}
\address{Cole Hawkins, Amherst College, Amherst, MA 01002}
\email{chawkins16@amherst.edu}

\author[Sutherland]{Alexander J. Sutherland}
\address{Alex Sutherland, Oberlin, OH 44074 Oberlin, OH 44074}
\email{asutherl@oberlin.edu}

\begin{document}

\begin{abstract}
We investigate decompositions of Betti diagrams over a polynomial
ring within the framework of Boij-S\"oderberg theory. 
That is, given a Betti diagram, we determine if it is possible to decompose it into the Betti diagrams of
complete intersections. 
To do so, we determine the extremal rays of the cone generated by the diagrams of complete intersections and provide a factorial time algorithm for decomposition.  

\end{abstract}

\thanks{This research was supported by NSF grant DMS-1460982.  The authors would like to acknowledge David Cox for his insightful suggestions for the decomposition algorithm.}
\maketitle

\section*{Introduction}

Betti diagrams of modules over graded rings are important objects of study in commutative algebra. 
Many numerical invariants can be read off of a Betti diagram; among others, this list includes the projective dimension, regularity, and the Hilbert series of the module.  
Motivated by \cite{GJMRSW}, we restrict our study to Betti diagrams that can be realized as a positive rational combination of Betti diagrams of complete intersection modules.  One reason for this study is that a decomposition of the Betti diagram of a module $M$ into a distinguished class of Betti diagrams can lead to surprising numerical information about $M$ itself that is not immediately evident from its Betti diagram; for an example of this phenomenon, we refer the reader to \cite{McCullough}.
To study these diagrams and decompositions, we use the lens of Boij-S\"oderberg theory.

In the current literature, the blanket term ``Boij-S\"oderberg theory'' refers to the study of the positive rational cone of Betti diagrams of modules over a ring.  
In \cites{boij-sod1}, Boij and S\"oderberg define the \defi{cone of Betti diagrams over a polynomial ring $R$} to be
\[
\BQ(R) = \left \{ \sum_{M} r_M \beta(M) \, \Bigg\vert \,\substack{ {M \text{ a f.g. $R$-module},\, r_M \in \Qpos,}\\{ \text{ and  finitely many $r_M$ are nonzero}}} \right \};
\]
i.e., it is the positive hull of the set of Betti diagrams of finitely generated $R$-modules. Central to the study of this cone is the notion of an extremal ray.  The \defi{ray generated by $\beta(M)$} refers to the positive hull of $\beta(M)$ (intuitively, it is an infinitely long vector in the direction of $\beta(M)$).  A ray $r$ is said to be \defi{extremal} in a cone $C$ provided $r$ is not in the positive hull of $C\setminus r$.  That is, $r$ is not a positive rational combination of other rays in the cone.

In their first paper, Boij and S\"oderberg conjectured a means of decomposing each Betti diagram of a finitely generated Cohen-Macaulay module over a polynomial ring into a positive rational linear combination of those diagrams necessary to generate the cone, and they proved this conjecture in low codimensions.
Eisenbud and Shreyer proved the conjectures in full generality in \cite{eis-schrey1}.  

Our study focuses on Betti diagrams of complete intersection modules over the polynomial ring $R = \kk[x_1,\ldots,x_n]$.  
A \defi{complete intersection $R$-module} is a quotient $M = R(-j)/(f_1,\ldots,f_c)$ where $f_1,\ldots,f_c$ form a homogeneous regular sequence.  
That is, $I = (f_1,\ldots,f_c)$ where $f_j$ is homogeneous for each $j$, $f_1$ is not a zero divisor on $R$, $f_\ell$ is not a zero divisor on $R(-j)/(f_1,\ldots,f_{\ell-1})$ for each $\ell$, and none of the quotients $R(-j)/(f_1,\ldots,f_\ell)$ are zero.
We denote the \defi{cone of Betti diagrams of complete intersection $R$-modules}, denoted $\BQCI(R)$, to be the subcone of $\BQ(R)$ defined by \[
\BQCI(R) =  \left \{ \sum_{M} r_M \beta(M) \, \Bigg\vert \, \substack{ {M \text{ a complete intersection $R$-module},}\\ {r_M \in \QQ_{\geq
    0}, \text{ and finitely many $r_M$ are nonzero}}} \right \}.
\]

The following result identifies the extremal rays of $\BQCI(R)$.

\begin{theorem}\label{thm:main1} Let $\kk$ be any field and let $R = \kk[x_1,\ldots,x_n]$.  A ray $r$ is extremal in the subcone $\BQCI(R)$ if and only if $r$ is generated by the Betti diagram a complete intersection $R$-module.
\end{theorem}

One implication is easy to see: if a ray is not generated by the Betti diagram of a complete intersection module, then the definition of $\BQCI(R)$ precludes it from being an extremal ray in $\BQCI(R)$.  The proof of the converse is the focus of \S\ref{sec:main}.

Assuming this theorem, we have that any Betti diagram in $\BQCI(R)$ decomposes into the Betti diagrams of complete intersections. Moreover, for each such Betti diagram in the decomposition, there is a shortcut for finding a decomposition into diagrams lying on extremal rays of larger cone $\BQ(R)$ up to codimension $3$ (see \cite{GJMRSW}) and some cases in codimension $4$ (see \cites{fanny,robert}). 
For example, 
let $M$ be the $R$-module described in Example~\ref{ex:cole}. 
This module is not a complete intersection.  
However, Example~\ref{ex:mikeex} gives a decomposition of $\beta(M)$ into the Betti diagrams of complete intersections:
\[\beta(M) = \beta\left(\kk[x,y]/(x^2,y^2)\right) + \beta\left(\kk[x,y]/(x^2,y^3)\right).\]
Via \cite{GJMRSW}*{Proposition 3.2}, we have a decomposition for the Betti diagrams on the right hand side, and hence we obtain a decomposition of $\beta(M)$ into diagrams spanned by extremal rays of $\BQ(R)$. 
Work on such formulas for higher codimensions is ongoing; as more cases are known, it becomes expedient to first decompose within $\BQCI(R)$ when possible, then use formulas for decomposing each Betti diagram of a complete intersection.  Our decomposition algorithm, Algorithm~\ref{algo:algo}, is implemented and available for examination at \cite{GibbonsWeb}.

We give a short introduction to Betti diagrams, especially those of complete intersections, in \S\ref{sec:background}.
In \S\ref{sec:main}, we prove our main result.  
In \S\ref{sec:decomp}, we provide a factorial time algorithm for deciding whether a diagram is a member of the cone.  
If so, the algorithm returns at least one decomposition of the diagram into a positive rational combination of Betti diagrams of complete intersections.
We outline further avenues of inquiry in \S\ref{sec:futurework}.

\section{Background}\label{sec:background}

The \defi{Betti diagram} of $M$, denoted $\beta(M)$, is defined to be the matrix for which $\beta_{i,i+j}(M)$ occurs in column $i$ and row $j$.
This indexing convention stems from the commutative algebra software Macaulay2 \cite{M2}.  
When we display a Betti diagram, the top left entry is $\beta_{0,0}(M)$ unless otherwise noted, and a zero in the table is represented by $-$; 
thus, when $M$ is a module finitely generated in degrees $0$ and higher over a polynomial ring $R$ in $n$ variables, Hilbert's Syzygy Theorem\footnote{For a full treatment, see \cite{eisenbook}*{Chapter 19} or \cite{peeva}*{Chapter 15}.} yields 
\[
\beta(M) = 
\scalebox{.7}{$
\begin{bmatrix}
    \beta_{0,0}(M) & \beta_{1,1}(M) &  \beta_{2,2}(M) & \cdots & \beta_{n,n}(M) & - & \cdots
    \\
    \beta_{0,1}(M) & \beta_{1,2}(M) &  \beta_{2,3}(M) & \cdots & \beta_{n,n+1}(M) & - & \cdots\\
    \vdots &\vdots & \vdots & \vdots & & \vdots & \\
     \beta_{0,j}(M) & \beta_{1,1+j}(M) &  \beta_{2,2+j}(M) & \cdots & \beta_{n,n+j}(M) & - & \cdots\\
      \vdots &\vdots & \vdots & \vdots & & \vdots & \\
  \end{bmatrix}$}.
\]  
The entries of such a diagram, called \defi{Betti numbers}, encode the ranks and twists of each free module in a minimal free resolution of $M$.  
We suppress the output of those columns consisting entirely of zeros for brevity.
The \defi{total Betti numbers of $M$} are the sums of the columns of $\beta(M)$;
the $i$-th total Betti number of $M$ is $\beta_i(M) = \sum_j \beta_{i,j}(M)$.

Let $\VV$ denote the $\QQ$-vector space space of column- and row-finite $\Zpos \times \ZZ$-indexed matrices with entries in $\QQ$, 
and let $\UU$ denote the subset of such matrices with integer entries.  
Betti diagrams of finitely generated $R$-modules are elements of $\UU$, and therefore elements of $\VV$.  
For $\gamma \in \VV$, the notation $\gamma_{i,i+j}$ refers to the entry in column $i$ and row $j$, as is the case with Betti diagrams. 
We abuse notation slightly and, for a diagram $\gamma \in \Vpos$, we define the \defi{projective dimension of $\gamma$} to be $\pdim{\gamma} = \max\{ i \, | \, \exists j \, : \, \gamma_{i,j+i} \not = 0\}$.  
In a similarly abusive fashion, we define the \defi{regularity of $\gamma$} to be $\reg(\gamma) = \max\{j \, | \, \exists i \, : \, \gamma_{i,j+i} \not = 0\}$ (since $\gamma \in \VV$ has finitely many nonzero entries, $\pdim(\gamma)$ and $\reg(\gamma)$ are both finite).
Intuitively, $\pdim(\gamma)$ is the width of $\gamma$ and $\reg(\gamma)$ is the height of $\gamma$.  
For calculations and examples, we tend to associate $\gamma$ with its image in the finite dimensional subspace $\QQ^{\pdim(\gamma)} \times \QQ^{\reg(\gamma)} \subseteq \VV$.

Let $M = R(-a_0)/(f_1,\ldots,f_c)$ be a complete intersection $R$-module.  We say that $c$ is the \defi{codimension of $M$}. 

\begin{remark}\label{rmk:Koszul} Let $a_\ell$ be the degree of $f_\ell$.  The minimal free resolution of $M$ is given by a twist of the Koszul complex $K(f_1,\ldots,f_c)_\bullet \otimes R(-a_0)$:  
\[0 \longleftarrow K_0 \longleftarrow K_1 \longleftarrow \cdots \longleftarrow K_c \longleftarrow 0, \]
where $K_0 = R(-j)$ and $K_i$ is the rank $\binom{c}{i}$ free module given by \[K_i = \bigoplus_{\text{distinct } \ell_k} R\left(-a_0-(a_{\ell_1} + \cdots +a_{\ell_i})\right).\]  
We refer the reader to \cite{eisenbook}*{Proposition 1.15} or \cite{peeva}*{Chapter 14} for a description of the differential of this complex; 
we omit it in this discussion because we do not use it directly.  
Observe that a complete intersection of codimension $c$ with twist $a_0$ and generators of degrees $a_1,\ldots,a_c$ has projective dimension $c$ and regularity $h =\sum_{k = 0}^c a_k - c$.  That is, the nonzero entries of its Betti diagrams occur within columns $0$ through $c$ and rows $0$ through $h$.
The symmetry of the binomial coefficients is reproduced in the total Betti numbers of the complete intersection, which is a useful observation.

It is also useful to note the Koszul complex can be realized as a tensor product of smaller Koszul complexes; indeed, $K(f_1,\ldots,f_c) = K(f_1,\ldots,f_{c-1})\otimes K(f_c)$. 
Since this multiplication is commutative, we may assume without loss of generality that the elements of the regular sequence $f_1,\ldots,f_c$ are listed so that their degrees are nondecreasing.  Although this may lead to different orderings of the forms themselves, it is not the precise forms in the homogeneous regular sequence that contribute to the numerics, but rather their degrees and the twist $j$.\end{remark}

Thus, we define the \defi{determining vector} of the complete intersection $M = R(-a_0)/(f_1,\ldots,f_c)$ to be the tuple $\underline a = (a_0,a_1,\ldots,a_c)$ where $a_\ell=\deg(f_\ell)$ for $\ell \geq 1$ and $a_1 \leq a_2 \leq \cdots \leq a_c$,  and we refer to its Betti diagram as $\beta(\underline a)$.

\begin{remark}\label{rmk:cibetti} The entry in column $i$ and row $j$ of the Betti table of a complete intersection with determining vector $\underline a = (a_0,a_1,\ldots,a_c)$ represents the number of collections of $i$ distinct entries $a_1,\dots,a_n$ that sum to $j+a_0$.

For example,
\[
\beta(1,2,2,3) = \begin{bmatrix}
- & - & - & - \\
1 & - & - & - \\
- & 2 & - & - \\
- & 1 & 1 & - \\
- & - & 2 & - \\
- & - & - & 1 \\
\end{bmatrix}.
\]
\end{remark}

\begin{example}\label{ex:cole} It's interesting to note that $\BQCI(R)$ contains rays that span Betti diagrams of modules that are not complete intersections or direct sums thereof.  
Indeed, consider $M  = \coker \scalebox{.7}{$\begin{pmatrix} xy-y^2 & x^2 + y^2 & -y^2 & 0 \\ -y^3 & y^3 & x^2y - y^3 & x^2 -y^2 \end{pmatrix}$}$.  
This module is not isomorphic to a direct sum of complete intersections as its presentation matrix cannot be written as a product of matrices $PAQ$ where $P$ and $Q$ are invertible, and where $A$ is a direct sum of row vectors (i.e., the presentation matrices of complete intersection modules), each of whose entries form a homogeneous regular sequence.  
Left multiplication by the invertible matrix $P$ amounts to a series of row operations on $A$ and right multiplication by $Q$ similarly amounts to a series of column operations on $A$.  
No series of operations on the matrix above yields such a matrix.

However, $\beta(M) = \beta(0,2,2) + \beta(0,2,3)$; we describe how to calculate such a decomposition in Section~\ref{sec:decomp}.
\end{example}

\section{Extremal rays in the subcone generated by complete intersections}\label{sec:main}

Let $\kk$ be a field and let $R = \kk[x_1,\ldots, x_n]$. 
In the following, we consider complete intersection $R$-modules of codimension $c \leq n$.

Let $\gamma$ be the Betti diagram of a codimension $c$ complete intersection $R$-module.  
In order to index a list of determining vectors, we use parenthesized superscripts. Theorem~\ref{thm:main1} concerns the equation
\begin{equation}\label{eq:good}
 \gamma=\sum\limits_{k=1}^{\ell} r_k\beta\left(a_0^{(k)},a_1^{(k)},\ldots,a_{c_k}^{(k)}\right) 
\end{equation}
where each coefficient $r_k$ is a positive rational number.  
To prove Theorem~\ref{thm:main1}, it is enough to show that only one distinct determining vector may appear on the right hand side.  Subsequently, by clearing denominators and re-indexing we can rephrase Theorem~\ref{thm:main1} as follows:

\begin{theorem}\label{thm:rephrasing} Let $\gamma$ be the Betti diagram of a complete intersection.  \\
For $p\in\mathbb{Z}_{\geq 0}$, if $p\gamma=\sum\limits_{k=1}^{p} \beta(\underline a^{(k)})$ for determining vectors $\underline a^{(1)}, \underline a^{(2)}, \dots, \underline a^{(p)}$ then $\beta(\underline a^{(k)})= \gamma$ for all $k$.
\end{theorem}

We record several results before proving this theorem.  

\begin{lemma}\label{lem:redux} Let $\gamma = \beta\left(a_0,a_1,\ldots,a_c\right)$.  Let $\underline{a}^{(k)}$ be as in the statement of Equation~\eqref{eq:good}.  Then, for all $k$, 
\begin{enumerate}
\item $a_0^{(k)} = a_0$ and $c_k = c$, and
\item $\displaystyle\sum_{j = 1}^c a^{(k)}_j = \displaystyle\sum_{j = 1}^c a_j$.
\end{enumerate}
\end{lemma}

\begin{proof} Without loss of generality, assume $a_0 = 0$.  
If $a_0^{(k)} \not = 0$, then $\gamma$ has two nonzero entries in column 0; this contradicts Remark~\ref{rmk:cibetti}. 
Thus, we may assume $a_0^{(k)} = 0$ for all $k$.  
If there exists $k$ for which $c_k > c$, then column $c_k$ of $\gamma$ has a nonzero entry; 
this contradicts that $c$ is the codimension of $\gamma$.  
Thus, we may assume $c_k \leq c$ for all $k$.  
Finally, if there exists $k$ for which $c_k < \gamma$, then $\gamma_{c,c+a_1+\dotsb+a_c} < p$; this contradicts Remark~\ref{rmk:cibetti}. 

For the second claim, note that if $\displaystyle\sum_{j = 1}^c a^{(k)}_j \neq \displaystyle\sum_{j = 1}^c a_j$, then $\gamma$ has two nonzero entries in column $c$; this is again a contradiction to Remark~\ref{rmk:cibetti}.
\end{proof}

In \cite{GJMRSW}, the authors define the operation $\odot$ on $\VV$ via \[
{(\alpha \odot \beta)}_{i,j} := \displaystyle\sum_{\substack{i_1 + i_2 = i \\ j_1 + j_2 = j}} \alpha_{i_1,j_1}\beta_{i_2,j_2}.
 \]

\begin{prop}\label{prop:ID}
The set $\UU$ is the additive closure of \[
\left\{\beta(M) \, | \, M \, \text{is an $R$-module}\}\cup\{-\beta(M) \, | \, M \, \text{is an $R$-module}.\right\}.
\] 
Moreover, $\UU$ is an integral domain.
\end{prop}

\begin{proof}  Given a Betti diagram $\beta$, formally define $-\beta$ by $(-\beta)_{i,j} = -({\beta}_{i,j})$.  
Assume that for each $1 \leq i \leq n-1$ we have a diagram $\delta^{(i)}$ such that ${\delta^{(i)}}_{i,i} = 1$ and all other entries of $\delta^{(i)}$ are $0$;  
when $i  = 0$, $\delta^{(i)} = \beta(R)$.  
Suppose $1 \leq c < n$.  
Given the determining vector $\underline a = (0,\underbrace{1,\dots,1}_{c+1 \, \text{times}})$, we have that \[
\beta_{i,j}(\underline a) = \begin{cases}
\binom{c+1}{i} & \text{if} \, j=i \\ 0 & \text{otherwise}.
\end{cases}
\]
It follows that $\delta^{(c+1)} = \beta(\underline a) - \displaystyle\sum_{i=0}^{c} \binom{c+1}{i}\delta^{(i)}$, so the claim holds by induction.\\ 
Note that $\left\{\delta^{(i)}\odot\beta(R(-j)) \, | \, i \in \mathbb{N}_0 \text{ and } j \in \mathbb{Z}\right\}$ is a basis for $\UU$.
This shows that $\UU$ is the additive closure of the desired set.

Evidently, $\UU$ is an abelian group under addition.  
Since the product of Betti diagrams has finitely many nonzero entries, it follows that this product is in $\UU$. 
Therefore, $\UU$ is closed under multiplication.  
It follows immediately from the definition of $\odot$ that $\beta(R)$ is the multiplicative identity of $\UU$ and that $\odot$ is commutative, associative, and distributive over Betti table addition.  

Finally, we show that $\UU$ has no zero divisors.  
Assume that there exist nonzero $\beta,\beta'\in \UU$ such that $\beta \odot \beta'=0$. 
Let $p$ be the smallest integer such that $\beta_{p+1,j}$ is $0$ for all $j$. 
Then, there exist nonzero entries in $\beta_{p,j}$. Let $q$ be the largest integer such that $\beta_{p,q}$ is nonzero. Define $p', q'$ similarly for $\beta'$.

By definition of $\odot$,
\begin{align*}
(\beta\odot{\beta'})_{p+p',q+q'} &= \displaystyle\sum_{\substack{i_1 + i_2 = p+p' \\ j_1 + j_2 = q+q'}} \beta_{i_1,j_1}{\beta'}_{i_2,j_2} \\
&= \beta_{p,q}{\beta'}_{p',q'} + \displaystyle\sum_{\substack{i_1 + i_2 = p+p' \\ j_1 + j_2 = q+q' \\  (i_1,j_1) \neq (p,q) \\ (i_2,j_2) \neq (p',q')}} \beta_{i_1,j_1}{\beta'}_{i_2,j_2} \\
&= \beta_{p,q}{\beta'}_{p',q'} \not = 0.
\end{align*}
\end{proof}

\begin{corollary}\label{cor:goodcor} The multiplication $\odot$ is cancellative; that is, whenever $\beta$ is nonzero and $\alpha \odot \beta = \omega \odot \beta$, we have that $\alpha = \omega$. \hfill $\qed$ \end{corollary}

The following fact is given in passing in \cite{GJMRSW}*{Definition 5.1}.
\begin{lemma}\label{lem:Koszul Lemma}
Let $\underline{a} = (a_0,a_1,\dots,a_n)$. Then \[\beta(\underline{a}) = \beta(a_0,a_1,\dotsc,a_{i-1},a_{i+1},\dotsc,a_n) \odot \beta(a_0,a_i).\]
\end{lemma}

We now prove Theorem \ref{thm:rephrasing}.

\begin{proof}[Proof of Theorem~\ref{thm:rephrasing}]  
We proceed by induction on the codimension $c$ of the complete intersection. 
Observe that when $c = 1$, the assertion follows immediately.  
The case $c =2$ is instructive, so we prove this as the basis for our induction.
 
Without loss of generality, we assume $\gamma = \beta(0,a_1,a_2)$ and for a fixed $i$ we write $(\gamma_{i})$ to denote the $i$-th column of $\gamma$.
Lemma \ref{lem:redux} allows us to assume that for each determining vector in Equation~\eqref{eq:good}, we have that $\underline{a}^{(k)} = (0,a^{(k)}_1,a^{(k)}_2)$.  
If $(\gamma_{1})$ has only one nonzero entry, then $a_1 = a_2$, and hence $a^{(k)}_1 = a^{(k)}_2$.  
Then, all determining vectors in the summation must  equal $(0,a_1,a_1)$ or else $(\gamma_{2})$ contains two nonzero entries. 

If $(\gamma_{1})$ has two distinct entries then $a_1 < a_2$.  
If determining vectors distinct from $(0,a_1,a_2)$ appear in our summation, then one of the following  pairs of determining vectors must appear in our summation: $(0,a_1,a_1)$ and $(0,a_1,a_2)$, $(0,a_1,a_2)$ and $(0,a_2,a_2)$, or $(0,a_1,a_1)$ and $(0,a_2,a_2)$.  
In all three cases, two nonzero entries must appear in the second column of $\gamma$.  

We now assume our theorem in the case of codimension $c-1$ where $1 < c \leq n$.

\begin{equation*}
p\gamma=\sum\limits_{k=1}^{p} \beta(a^{(k)})
\end{equation*}

Lemma~\ref{lem:redux} allows us to assume that each determining vector in Equation~\eqref{eq:good} has $c_k = c$ and $a_0^{(k)} = 0$.  
Let $\beta^{(k)}=\beta(\underline a^{(k)})$.

For a given determining vector $\underline a^{(k)}$ of length $c+1$ we will refer to the vector $(a_1^{(k)},\dots, a_c^{(k)})$ as the corresponding \defi{degree vector}.  
We will now show that the degree vector of each $\underline a^{(k)}$ must share one common entry and use Corollary~\ref{cor:goodcor} and Lemma~\ref{lem:Koszul Lemma} to reduce to the $c-1$ case.  
Let $a_{1}$ be the minimum entry present in any degree vector of a determining vector in the summation.  
Let $m$ be the maximum number of times that $a_{1}$ appears in any degree vector.  
We recall from Remark~\ref{rmk:cibetti} that $\beta_{m,m\cdot a_1-m}^{(k)}$ represents the number of distinct sums of $m$ entries in the degree vector $\underline a^{(k)}$ that add to $m\cdot a_1$.  
If all degree vectors in the summation share entry $a_1$, then we are done.  
Otherwise, assume the degree vector of determining vector $\underline a^{(s)}$ does not contain $a_1$.  
Then $\beta(a^{(s)})_{m,m\cdot a_1-m}=0$.  
Since $m$ is a maximum, there are at most $p-1$ determining vectors $\underline a^{(k)}$ whose Betti diagrams can contribute at most $1$ to $\gamma_{m,m\cdot a_1-m}$.  
However, scaling by $p$ leaves a fractional entry in $\gamma_{m,m\cdot a_1-m}$, contradicting Remark~\ref{rmk:cibetti}.
Therefore, each $a^{(k)}$ shares entry $a_1$, and our proof is complete.  \end{proof}

\section{Decomposition algorithm}\label{sec:decomp}

In this section, we are concerned with arbitrary elements in the positive orthant of $\VV$, denoted $\Vpos$, because $\Vpos$ contains $\BQCI(R)$.  Elements of $\BQCI(R)$ are positive rational combinations of the Betti diagrams of complete intersections.  However, they need not have unique representation as such.

\begin{example}\label{ex:doubledecomp}
Some diagrams have multiple distinct decompositions into Betti diagrams of complete intersections.  We leave it to the reader to verify that $\gamma = \beta(0,2,2,2) + \beta(1,2,2,3)$ also decomposes as $\gamma = \beta(0,2,2,4) + \beta(1,1,2,2).$
These two decompositions lead to infinitely many decompositions of the form \[
\gamma = 
\frac{p}{q}\bigg(\beta(0,2,2,2) + \beta(1,2,2,3)\bigg) + \frac{q-p}{q}\bigg(\beta(0,2,2,4) + \beta(1,1,2,2)\bigg)
\] where $p \leq q \in \ZZ_{>0}$.
\end{example} 

The main goal of this section is to provide an algorithm for deciding when a diagram $\gamma$ is a member of $\BQCI(R)$ and, when it is, what its decompositions are. 

\begin{lemma}\label{lem:denominators} For each finite set of determining vectors $\{a^{(1)},\ldots,a^{(s)}\}$, there exists a positive integer $D$ such that, if $\gamma = \sum_{k = 1}^s a_k \beta(a^{(k)})$ for $a_k \in \Qpos$, then $D \gamma = \sum_{k = 1}^s b_k \beta(a^{(k)})$ where $b_k \in \Zpos$.
\end{lemma}

\begin{proof}  If $\gamma$ has fractional entries, multiply by $d$, the least common multiple of their denominators. Now $d\gamma \in \Upos$.

Identify $\beta(a^{(i)})$ with the column vector $\vec{v}_i$ of height $N = \pdim(\gamma) \times \reg(\gamma)$ in the intuitive way; similarly, identify $d\gamma$ with the column vector $\vec{v}$. Suppose $\vec{v} = a_1 \vec{v}_1 + \cdots + a_s \vec{v}_s$ for some $a_k \in \Qpos$. 
We use Carath\'eodory's theorem\footnote{We refer the reader to \cite{ziegler}*{Proposition 1.15} for a constructive proof of this result. Although the proposition is stated for $\mathbb{R}^N$, the proof goes through over $\mathbb{Q}^N$ as written.} to find linearly independent subset $\{\vec{v}_1,\dots,\vec{v}_t\}$ (after reindexing) of which $\vec{v}$ is a positive rational combination.  
We require linear independence in order to expand this set to a basis of $\QQ^N$ with standard basis vectors $\vec{e}_{k_1},\ldots,\vec{e}_{k_{N-t}}$.  
Then there exist coefficients $c_1,\ldots,c_N\in \QQ_{\geq 0}$ with $c_{t+1}=\dots = c_{N}=0$ for which \[
\vec{v} = c_1 \vec{v}_1 + \cdots + c_t \vec{v}_t + c_{t+1} \vec{e}_{k_1} + \cdots +c_{N}\vec{e}_{k_{N-t}}.
\]  Now, define the matrices 
\begin{align*}
A &= \begin{pmatrix} \vec{v}_1 & \cdots & \vec{v}_s \,\, | \,\, \vec{e}_{1} & \cdots & \vec{e}_{N} \end{pmatrix} \text{ and }\\
    B &= \begin{pmatrix} \vec{v}_1 & \cdots & \vec{v}_t \,\, | \,\, \vec{e}_{k_1} & \cdots & \vec{e}_{k_{N-t}} \end{pmatrix}.
\end{align*}
Let $d' \in \ZZ$ be the (positive) least common multiple of the nonzero $N\times N$ minors of $A$.  By construction, $B$ is an $N \times N$ submatrix of $A$ with linearly independent columns; hence $\det(B)$ is a factor of $d'$. 

We apply Cramer's rule to $B \vec{c} = \vec{v}$ to find that $c_k = \frac{\det(B_k)}{\det(B)}$. 
Define $b_k = d' c_k$ and observe that $b_k \in \Zpos$ because $\det(B)$ is a factor of $d'$, $d' \in \Zpos$, and $c_k$ is non-negative by assumption.
Hence \[
d'\vec{v} = b_1 \vec{v}_1 + \cdots + b_s \vec{v}_s
\]
is a non-negative, integral combination of $\vec{v}_1, \ldots, \vec{v}_s$. Letting $D = dd'$, we see that $D \gamma = \sum_{k=1}^s b_k \beta(a^{(k)})$, as desired.
\end{proof}

We provide an example to illustrate how to calculate $D$ for a diagram in $\BQCI(R)$.

\begin{example}\label{ex:findingD} Consider the diagram 
$\gamma = \begin{bmatrix} 1 &
 \scalebox{.8}{$
1/2$
} \\ 1 & \scalebox{.8}{$
3/2$
} \end{bmatrix} = \frac{1}{2}\beta(0,1)+\frac{1}{2}\beta(0,2)+\beta(1,1)$.  In order to use Lemma~\ref{lem:denominators}, we find the least common multiple of the denominators is $d = 2$.  Now we consider $2 \gamma = \begin{bmatrix}2 & 1 \\ 2& 3\end{bmatrix}$.

For each $\beta \in \{\beta(0,1),\beta(0,2),\beta(1,1)\}$, we identify $\beta$ with the column vector $[\beta_{0,0} \, \beta_{0,1} \, \beta_{1,1} \, \beta_{1,2}]^\top \in \QQ^4$.
Now \[
A = \left(\begin{array}{@{}ccc|cccc@{}}
1 & 1 & 0 & 1 & 0 & 0 & 0 \\ 
0 & 0 & 1 & 0 & 1 & 0 & 0 \\ 
1 & 0 & 0 & 0 & 0 & 1 & 0 \\
0 & 1 & 1 & 0 & 0 & 0 & 1
\end{array}\right).
\]

Since every nonzero $4\times4$ minor is $\pm1$, $d' = 1$.  Thus $D= dd' = 2$.
\end{example}

\begin{remark} In general, it is possible for $d' > 1$ to occur.  Consider \[\gamma = \frac{1}{4}\beta(0,2,2,2)+\frac{1}{4} \beta(1,2,2,3) + \frac{3}{4}\beta(0,2,2,4) +\frac{3}{4}\beta(1,1,2,2)\] (cf.\ Example~\ref{ex:doubledecomp}). Here, $\pdim(\gamma) = 4$ and $\reg(\gamma) = 6$, so $N = 24$.  We leave it to the reader to verify that $d = 1$ and $d' = 12$.\end{remark}

In order to check that a diagram $\gamma$ is within $\BQCI$, we may assume that we know several things.  
Indeed, we may assume that we know $\pdim(\gamma)$ and $\reg(\gamma)$ in order to input $\gamma$ into the algorithm;
we may assume that $\gamma$ has top left entry in row $0$ (otherwise, we instead consider $\gamma(-j)$ for appropriate $j$); 
and we may assume $\gamma \in \Upos$ by clearing denominators since $\gamma$ has finitely many nonzero, rational entries (as in the proof of Lemma~\ref{lem:denominators}).

\begin{algo}\label{algo:algo} Let $R = \kk[x_1,\ldots,x_n]$.
\begin{description}
\item[Input] An element $\gamma \in \Upos$ of width $w := \pdim(\gamma)$ (where $w \leq n$) and height $h := \reg(\gamma)$.
\item[Ouptut] All valid decompositions of $\gamma$ into the Betti diagrams of complete intersections (up to rational recombinations).  If no decompositions exist, then $\gamma \not \in \BQCI(R)$.
\item[Step 1]
Construct the lists 
\[C_0 := \{j \in \ZZ \, | \, \gamma_{0,j} \not = 0\}, \qquad C_1 := \{j \in \ZZ \, | \, \gamma_{1,1+j} \not = 0\}.\]
\item[Step 2]
Create the empty list $L$.  For each $j \in C_0$, for each $c$ for which $1 \leq c \leq w$, for each distinct determining vector $\underline{a} = (j,a_1,\ldots,a_c)$ where $j + a_i \in C_1$ and $j+\sum_{i = 1}^c a_i \leq h + c$
for all $i$, append $\underline{a}$ to $L$.
\item[Step 3]
Index the entries of $L$ so that $L = \{\underline{a}^{(1)},\ldots,\underline{a}^{(r)}\}$.  Define $D$ to be the positive integer from Lemma~\ref{lem:denominators} associated to $\beta(\underline{a}^{(1)}),\ldots,\beta(\underline{a}^{(r)})$.
\item[Step 4]
Define the parameter $m := \sum_{j=0}^h (D\gamma)_{0,j}$.
\item[Step 5]
For each $(s_1,\ldots,s_r) \in \Zpos^r$ satisfying $\sum_{i =1}^r s_i = m$, test whether $D\gamma = \sum_{i = 1}^r s_i \beta(\underline{a}^{(i)})$.  Whenever equality occurs, output the decomposition $\gamma = \frac{1}{D}\sum_{i = 1}^r s_i \beta(\underline{a}^{(i)})$.  If equality never occurs, output {\tt NULL}.
\end{description}
\end{algo}

Before presenting the proof that the algorithm terminates in factorial time and is correct, we provide a clarifying example.

\begin{example}\label{ex:algoex} Let $\gamma = \begin{bmatrix} 2 & 1 \\ 2 & 3 \end{bmatrix}$.  In Step 1, we populate the lists $C_0$ and $C_1$ by examining columns $0$ and $1$ respectively: $C_0 = \{0,1\}$, $C_1 = \{1,2\}$.  We also define the parameters $w$ and $h$, the width and height of $\gamma$, via $w = \pdim(\gamma) = 1$ and $h = \reg(\gamma) = 1$.

In Step 2, we populate $L$.  First, when $j = 0$, we append $(0,1)$ and $(0,2)$.  Next, when $j = 1$, we append $(1,1)$; note that we do not append $(2,1)$ because $2+1=3$ is not in the list $C_1$.  There are no vectors that are discarded for exceeding the height of $\gamma$ in this example, though this may happen.

In Step 3, we obtain $D = 1$.

In Step 4, we calculate $m = 4$.  Since $L$ has length $3$, we consider all triples $(s_1,s_2,s_3) \in \Zpos$ for which $s_1 + s_2 + s_3 = 4$.

In Step 5, we test whether $D\gamma = s_1 \beta(0,1) + s_2 \beta(0,2) + s_3 \beta(1,1)$.  There are fifteen such tuples to verify (we leave the discovery and verification of all such tuples the reader).  We then output all of the successful combinations for $D\gamma$; there is only one, and it is \[
1 \beta(0 , 1) + 1\beta(0 , 2) + 2 \beta(1 , 1).
\]
\end{example}

\begin{prop}\label{prop:runtime} Algorithm~\ref{algo:algo} runs in factorial time.
\end{prop}

\begin{proof} Observe that Step 1 finishes in finite time because $\gamma$ has finitely many nonzero entries within columns $0$ through $h$ and rows $0$ through $w$ so that populating $C_0$ and $C_1$ occurs in finite (linear) time. 

There are finitely many determining vectors that we consider adding to $L$ in Step 2.  In order to apply the criteria described in Step $2$ we consider at least $\sum_{i=1}^{h+1}\sum_{j=1}^{w+1} \binom{|C_1|}{j}$ determining vectors.  The worst case bound on $|C_1|$ is $h+1$.

Since $L$ is a finite list, the calculation of $D$ in Step 3 terminates in finite time.  Let $z = (w+1)(h+1)$, and recall that $r = |L|$.  Calculating the minors of a $z\times(r+z)$ matrix by expansion is proportional to $z^2\binom{r+z}{z}$.  The least common multiple calculation is possible in polynomial time using the Euclidean algorithm.

The number of tuples in Step 4 is finite (there are $\binom{r+m -1}{m-1}$ such tuples).  Checking equality for each of these tuples takes time proportional to $\binom{m2^n}{m}$, where $n = \sum_{j = 1}^{h} D\gamma_{1,1+j}$.

The combinatorial terms that describe Steps $2,3, $ and $4$ mean that the algorithm runs in time factorially proportional to the height and width of the input diagram. 
\end{proof}

Note that this algorithm is much slower than the original Boij-S\"oderberg decomposition algorithm which runs in polynomial time proportional to the height and width of $\gamma$.  Indeed, the pure diagram required at each step of that algorithm is prescribed by the ``top line'' of the diagram resulting from subtracting the rational multiple of the pure diagram obtained in the previous step. Finding the appropriate rational coefficient is a greedy process bounded by the codimension of $\gamma$.

Now we show that the algorithm is correct.
\begin{prop}\label{prop:correct} Algorithm~\ref{algo:algo} produces a positive rational decomposition of $\gamma$ if and only if $\gamma \in \BQCI(R)$.
\end{prop}

\begin{proof} Suppose the input $\gamma \in \Upos$ has a decomposition using the Betti diagrams of complete intersections.  Our algorithm should produce all possible decompositions up to rational linear combinations resulting in positive coefficients (see Example~\ref{ex:doubledecomp}), so it is enough to show that it produces all decompositions into linearly independent Betti diagrams of complete intersections.  We invoke Carath\'eodory's theorem again to assert that there exists a linearly independent subset of the diagrams in the decomposition of $\gamma$ so that \begin{equation}\label{eq:gammadec}
\gamma = \sum_{k = 1}^\ell q_k \beta(\underline a^{(k)})
\end{equation}
for strictly positive $q_k \in \Qpos$ and linearly independent $\beta(\underline a^{(k)})$. We show that the algorithm outputs this decomposition of $\gamma$. 

First we show that, after Step 2 is complete, $\underline a^{(k)} \in L$ for $k = 1,\ldots,\ell$. 

Since $q_k > 0$ for each $k$, $\gamma_{0,j} \not = 0$ whenever $\beta(\underline a^{(k)})_{0,j} \not = 0$.  
The latter occurs precisely when $j = \underline a^{(k)}_0$.  Thus $a^{(k)}_0 \in C_0$ after Step 1 is complete.
Similarly, $\gamma_{1,1+j} \not = 0$ whenever $\beta(\underline a^{(k)})_{1,j+1} \not = 0$.  
The latter occurs precisely when $j = \underline a^{(k)}_0 + \underline a^{(k)}_1$, and hence $\underline a^{(k)}_0 + \underline a^{(k)}_1 \in C_1$ after Step 1 is complete.

Let $c_k = \codim(\beta(\underline a^{(k)}))$ and $h_k = \reg(\beta(\underline a^{(k)}))$. Equation~\eqref{eq:gammadec} implies that $c_k \leq w$ and $h_k \leq h$ for each $k$.  
By Remark~\ref{rmk:Koszul}, $\beta(\underline a^{(k)})_{1,1 + a^{(k)}_0 + a^{(k)}_1} >0$, and thus $\gamma_{1, 1 + a^{(k)}_0 + a^{(k)}_1} > 0$. 
By the same remark, $h_k = a^{(k)}_0 + \cdots + a^{(k)}_{c_k} - c_k$.
This means that the desired criteria hold, and thus $\underline a^{(k)} \in L$ after Step 2 as claimed.

Next, given $L$, we show that the algorithm outputs Equation~\eqref{eq:gammadec}.

In Step 3, we calculate $D$ for Betti diagrams of determining vectors in $L$ via Lemma~\ref{lem:denominators}.
In Step 4, we define $m$ to be the sum of the entries in the $0$-th column of $D\gamma$.  Recall that the Betti diagram of a complete intersection has a single $1$ in column $0$.
Subsequently, if $D\gamma$ can be written as a positive integral combination of Betti diagrams of complete intersections, then there must be $m$ of them (not necessarily distinct). 
We then perform an exhaustive search all sums of $m$ Betti diagrams of complete intersections with determining vectors in $L$.  

By the construction of $D$ in Lemma~\ref{lem:denominators}, we have that $D q_k \in \Zpos$ for each $k$ (and that $\ell \leq m$), and therefore $D\gamma = \sum_{k = 1}^\ell D q_k \beta(\underline a^{(k)})$ is a positive integral combination that is output by Step 5 as desired.

Conversely, if $D\gamma$ is a positive, integral combination of Betti diagrams of complete intersections, it is evident that $\gamma$ is a positive, rational combination of the same Betti diagrams of complete intersections (just divide by $D$).\end{proof}

\begin{example}\label{ex:mikeex}
Recall from Example~\ref{ex:cole} the module $M$ with 
\[\beta(M) = \begin{bmatrix}
    2 & - & - \\
    - & 3 & -  \\
    - & 1 & 1  \\
    - & - & 1  \\ 
     \end{bmatrix}.\]
Algorithm~\ref{algo:algo} produces the decomposition $\beta(M) = \beta(0,2,2) + \beta(0,2,3)$, and hence $\beta(M) \in \BQCI(R)$.
\end{example} 

\section{Future work}\label{sec:futurework}

In their seminal work \cite{boij-sod1}, Boij and S\"{o}derberg construct a partial ordering on the extremal rays of the cone generated by all Betti diagrams.  
Then, they demonstrate that this ordering induces a simplicial structure on the cone.
It is from this simplicial structure that they produce an algorithm to decompose an element of the cone into a rational combination of maximal chain elements that correspond to extremal rays. 
In fact, we attempted to use the same partial ordering on determining vectors to produce a decomposition algorithm. We recall the definition of this partial order:
\begin{center}
For two determining vectors $\underline a = (a_0, a_1, \dots , a_n)$ and $\underline b = (b_0, b_1, \dots , b_n)$, \\
$\underline a \leq \underline b$ if $a_i \leq b_i$ for $0\leq i \leq n$. 
\end{center}
For us, this  choice of partial ordering has some nice properties. 
Indeed, all elements in a chain are linearly independent in $\VV$, and a proof of this closely follows the proof of \cite{boij-sod1}*{Proposition 2.9}.  
However, $\BQCI(R)$ need not be the union of spans of the maximal chains, as is demonstrated in the next example.

\begin{example}\label{ex:notchains}
Consider the diagram \[
\beta = \begin{bmatrix}
    2 & 1 & - & - \\
    - & - & - & - \\
    - & 2 & - & - \\
    - & 2 & 1 & - \\ 
    - & 1 & 2 & - \\
    - & - & 2 & - \\
    - & - & - & - \\
    - & - & 1 & 2\\
     \end{bmatrix}.\]
Examining $\beta$, we may assume that any determining vectors of diagrams in its decomposition have codimension 3 since the total Betti numbers $\beta_0 = \beta_3 = 2$.  
Furthermore, since the first and last columns have precisely one entry each, the determining vectors both have twist $0$ and degrees that sum to $10$. 
Using the first two steps of Algorithm~\ref{algo:algo} to populate $L$, we find that the possible determining vectors in $L$ that satisfy these additional criteria are $(0,1,4,5)$ and $(0,3,3,4)$.\footnote{
	Without using these observations, calculating $D$ requires finding the $32$-minors of a $32\times56$ matrix.  
	Even using ``fast'' determinant algorithms for each minor calculation, there are $\binom{56}{32}$ of these calculations to do.
	}  
In fact, $\beta = \beta(0,1,4,5) + \beta(0,3,3,4)$, so $\beta \in \BQCI(R)$. These two determining vectors are incomparable by the partial ordering above, so they cannot belong to the same chain. 
\end{example}

Nevertheless, we believe that there may be a partial ordering on the extremal rays of $\BQCI(R)$ that allows us to realize the geometric structure of the cone.

\begin{question} Is the cone $\BQCI(R)$ a geometric realization of a poset as in \cite{boij-sod1}?
\end{question}

\end{document}